\documentclass[11pt,a4paper]{article}

\usepackage{a4wide}
\usepackage{graphicx}
\usepackage{latexsym}
\usepackage{epsfig}
\usepackage{amssymb}
\usepackage{amstext}
\usepackage{amsgen}
\usepackage{amsxtra}
\usepackage{amsgen}
\usepackage{amsthm}
\usepackage{cite}
\usepackage{enumerate,amsmath,subfigure,color}

\newtheorem{thm}{Theorem}[section]
\newtheorem{prop}[thm]{Proposition}
\newtheorem{lemma}[thm]{Lemma}
\newtheorem{cor}[thm]{Corollary}
\newtheorem{rem}[thm]{Remark}

\theoremstyle{definition}

\newtheorem{definition}[thm]{Definition}

\numberwithin{equation}{section}

\newcommand{\N}{\mathbb{N}}

\newcommand{\R}{\mathbb{R}}

\newcommand{\lk}{\left(}
\newcommand{\rk}{\right)}

\newcommand{\na}{\nabla}
\newcommand{\pa}{\partial}

\newcommand{\Hdual}{(H^{1})'}
\newcommand{\Hdualm}{(H^{1}_\diamond)'}

\begin{document}

\title{Delayed Blow-Up for Chemotaxis Models with Local Sensing}
\author{Martin Burger\thanks{Department Mathematik, Friedrich-Alexander Universit\"at Erlangen-N\"urnberg, Cauerstr. 11, 91058 Erlangen, Germany. email: martin.burger@fau.de} \and Philippe Lauren\c cot\thanks{Institut de Math\'ematiques de Toulouse, UMR 5219, Universit\'e de Toulouse, CNRS, 31062 Toulouse Cedex~9, France. emails: philippe.laurencot@math.univ-toulouse.fr, ariane.trescases@math.univ-toulouse.fr}
\and Ariane Trescases\footnotemark[\value{footnote}]}
\date{}
\maketitle

\begin{abstract}
The aim of this paper is to analyze a model for chemotaxis based on a local sensing mechanism instead of the gradient sensing mechanism used in the celebrated minimal Keller-Segel model. The model we study has the same entropy as the minimal Keller-Segel model, but a different dynamics to minimize this entropy. Consequently, the conditions on the mass for the existence of stationary solutions or blow-up are the same, however we make the interesting observation that with the local sensing mechanism the blow-up in the case of supercritical mass is delayed to infinite time. 

Our observation is made rigorous from a mathematical point via a proof of global existence of weak solutions for arbitrary large masses and space dimension. The key difference of our model to the minimal Keller-Segel model is that the structure of the equation allows for a duality estimate that implies a bound on the $\Hdual$-norm of the solutions, which can only grow with a square-root law in time. This additional $\Hdual$-bound implies a lower bound on the entropy, which contrasts markedly with the minimal Keller-Segel model for which it is unbounded from below in the supercritical case. 

Besides, regularity and uniqueness of solutions are also studied.

\end{abstract}

%\tableofcontents

\section{Introduction}

The purpose of this paper is to study the well-posedness and some qualitative properties of the following system arising in the modeling of chemotaxis:
\begin{subequations}\label{sys:localsensing}
\begin{align}
\partial_t u &= \Delta (\gamma(v) u), & \text{in } (0,\infty)\times \Omega, \label{PLe1}\\
\partial_t v &= \epsilon \Delta v - \beta v + u, &\text{in } (0,\infty)\times \Omega,\label{PLe2}
\end{align}
on a bounded domain $\Omega$ of $\R^d$ ($d\ge1$) supplemented with no-flux boundary conditions 
\begin{align}
\pa_\nu [\gamma(v) u] = \pa_\nu v =0, \qquad \text{ on } (0,\infty)\times \pa \Omega, \label{Neumann}
\end{align}
where $\nu=\nu(x)$ is the outward unit normal vector field at the boundary $\pa \Omega$, and nonnegative initial data 
\begin{align}
u(0,x)=u_{0}(x)\ge0, \qquad v(0,x)=v_{0}(x)\ge0, \qquad \text{ in } \Omega. \label{initial}
\end{align}
Here, $u=u(t,x)\ge0$ is the density of cells (or other organisms) producing a certain chemoattractant (or other field), its concentration being given by $v=v(t,x)\ge0$.
The cell motility $\gamma$ is given by
\begin{align}
\gamma(s)=e^{-s}, \qquad \text{for} \, s\ge 0 \label{PLe3},
\end{align}
while the parameters $\epsilon$ and $\beta$ are two positive constants representing the diffusivity of the chemoattractant and its degradation rate, respectively.
\end{subequations}

Mathematical models of chemotaxis have been studied extensively in the last decades, based on several variants of the classical Patlak-Keller-Segel model (cf. \cite{patlak1953random,keller1971Model}), and we refer to \cite{BBTW2015,stevens1997aggregation,MR2448428,horstmann20031970,Perthame} for an overview of the modelling and basic results. The system \eqref{PLe1}--\eqref{PLe2} is a particular case of the well-studied system introduced by Keller and Segel in their series of three seminal papers (cf. \cite{keller1970initiation,keller1971Model,keller1971TW}). In \cite{keller1971Model}, the equation they propose to model the evolution of the cell density is the following
\begin{equation*}
\partial_t u = \na \cdot (\mu(v) \nabla u - u \chi(v) \nabla v), \qquad \text{in } (0,\infty)\times \Omega,
\end{equation*}
where the cell diffusivity $\mu$ and the chemosensitivity $\chi$ are linked through
\begin{equation*}
\chi(v) = (\alpha-1)\mu'(v).
\end{equation*}
The coefficient $\alpha$ is a parameter related to the distance between the receptors in the cells with a suitable scaling. Equation~\eqref{PLe1} is recovered when $\alpha$ is taken equal to zero. This choice corresponds to the case where there is a single receptor in the cell. Therefore, the cell is not able to perceive a gradient in the concentration by comparing the concentrations at two different spots (this is gradient sensing), it can only sense the local concentration at the point where it is located (this is local sensing). Let us mention the discussion of different sensing mechanisms in \cite{stevens1997aggregation} related to microscopic descriptions of gradient and local sensing, which lead to the derivation of the continuum models mentioned above. Let us also mention the models for direct aggregation of animals (cf. \cite{burger2013individual}) and social segregation dynamics (cf. \cite{burger2018mean}), which are based on similar paradigms. The system \eqref{PLe1}--\eqref{PLe2} (more precisely a more general version, possibly adding a logistic growth term in the cell equation) has been recently used in the Biophysics literature (cf. \cite{Fuetal}) in order to describe the formation of stripe patterns in colonies of bacteria observed in experiments (cf. \cite{Liuetal}). In this study, the cell motility $\gamma$ is taken as a sharply degenerating function of the concentration (more precisely, a regularized version of step function decreasing to zero). Beyond simplicity, this choice is motivated by the experimental data showing that the cell motility gets close to zero above a certain threshold value of the chemoattractant concentration.

To illustrate the difference between the local sensing model and classical models, let us focus on the spatial dimension two. A key property of the minimal Patlak-Keller-Segel model is the ability to describe a mass-dependent threshold phenomenon: 
Cell aggregates only form if there is sufficient mass in the system, which is mathematically described by a blow-up of the solution. This was studied in particular for the model
\begin{equation}
\begin{split}
\partial_t u &= \nabla \cdot (\nabla u - u \nabla v), \qquad \text{in } (0,\infty)\times \Omega,\\
\partial_t v &= \epsilon \Delta v - \beta v + u, \quad \qquad \text{in } (0,\infty)\times \Omega,
\end{split}\label{mksm}
\end{equation}
and for its quasi-stationary version (with the second equation replaced by its elliptic counterpart). For these models, there is global existence and boundedness of solutions for small mass, whereas finite time blow-up could be shown as soon as the mass exceeds a critical threshold (cf. \cite{jager1992explosions,BN1993,herrero1997blow,nagai1998global}). It is hence an interesting question to investigate whether the system \eqref{PLe1}, \eqref{PLe2} has the same behaviour. This question is even more motivated from the fact that both models share the same entropy (or Lyapunov functional), namely 
\begin{equation}\label{def:entropy}
{\cal E}(u,v) :=  \int ( \frac{\epsilon}2  \vert \nabla v \vert^2  + \frac{\beta}2 \vert  v \vert^2 - u v + u \log u - u +1)~dx.
\end{equation}   
The minimal Keller-Segel model \eqref{mksm} is a gradient flow of the form
\begin{align*}
\partial_t u &= \nabla \cdot (u \nabla \partial_u {\cal E} ), & \text{in } (0,\infty)\times \Omega,\\
\partial_t v &= - \partial_v {\cal E}, & \text{in } (0,\infty)\times \Omega,
\end{align*}
for which the entropy has the dissipation (formally)
$$ \frac{d}{dt}{\cal E}(u(t),v(t)) = - \int_\Omega u |\nabla \log ( u e^{-v} )|^2~dx -  \int_\Omega |\epsilon \Delta v + u - \beta v|^2~dx.   $$
The system \eqref{PLe1}, \eqref{PLe2} is of the form
\begin{align*}
\partial_t u &= \Delta (e^{\partial_u {\cal E}} - 1 ), \qquad \text{in } (0,\infty)\times \Omega,\\
\partial_t v &= - \partial_v {\cal E}, \qquad \text{in } (0,\infty)\times \Omega,
\end{align*}
and a simple computation (formally) gives the dissipation
$$ \frac{d}{dt}{\cal E}(u(t),v(t)) = - \int_\Omega 4 |\nabla \sqrt{e^{-v}u}|^2~dx -  \int_\Omega |\epsilon \Delta v + u - \beta v|^2~dx.   $$
This generalized gradient structure also motivates the specific choice of $\gamma$ in \eqref{PLe3}. The natural assumptions the motility $\gamma$ should satisfy from the point of view of the modeling is that it should be typically a nonnegative decreasing function decaying to zero, the above generalized gradient structure is obtained only for an exponential dependence, which also corresponds to the use of Boltzmann statistics in the microscopic local sensing models. 

Based on the properties of the entropy, which is bounded below only for the case of subcritical mass (which is related to logarithmic functional inequalities, cf. \cite{DoPe2004}), it is not surprising that there is a similar threshold behaviour for \eqref{PLe1}, \eqref{PLe2}, in particular the existence of bounded solutions in the subcritical regime. The latter has recently been established in terms of classical solutions in \cite{jin2019keller} and will be further developed for weaker solutions below. The surprising result we obtain in this paper is however a global existence result for arbitrary masses (which actually holds for all dimensions). This implies that the blow-up is delayed compared to the minimal Keller-Segel model, it arises in the infinite time limit instead of finite time. From a modelling point of view it confirms the slower evolution via local sensing compared to the gradient sensing mechanism. From the point of view of mathematical analysis, the form \eqref{PLe1}, sometimes referred to as Laplace form, allows the use of some duality techniques which give an estimate that is crucial for the global existence. More precisely, this duality estimate allows us to derive a lower bound for the entropy (for finite time). Such estimate and its interplay with the entropy will be explained at the formal level in the next section.

Infinite-time blow-up has been identified in other variants of the classical Keller-Segel system, but is then due to different mechanisms. For instance, in \cite{CieslakStinner2012,CieslakStinner2015}, the first equation is nonlinear in $u$ and linear in $v$; that is, the motility and the chemosensitivity are density-dependent (but independent on the chemical concentration), which models a volume-filling effect. This volume-filling effect slows down the aggregation and, if strong enough and suitably tuned, can finally prevent finite-time blow-up without implying boundedness of the solution. In our study, the nonlinearity in $v$ (together with the specific laplacian-form) accounts for the nonlinear chemosensitivity of the organism, which models the saturation of the receptor(s) of the organism. In particular there is no volume-filling effect. Note also that our model has the same entropy as the original Keller-Segel model, it can be viewed as the natural counterpart when replacing the mechanism of gradient sensing by local sensing. Another example of delayed blow-up may be found in \cite{TaoWinkler_critical,Laurencot_indirect} and is due in that case to the indirect production of the chemoattractant; that is, the chemoattractant is no longer produced by the cells, but by an intermediate species.

\subsection{Main results}

Our first two results concern the well-posedness of the system in any dimension. For this we first introduce some notations and the notion of solution that we use.

\paragraph{Notations.} In the following, we consider a bounded domain $\Omega$ in $\R^d$, that we assume of measure one ($|\Omega|=1$) for simplicity. We denote by ${\bf 1}$ the constant function on $\Omega$, and for a function space $V$ on $\Omega$ such that ${\bf 1}$ is in the predual space of $V$, we denote by $V_\diamond$ the subspace of elements of $V$ with mean value zero, i.e. 
$$ V_\diamond = \{ v \in V~|~\langle v , {\bf 1} \rangle = 0 \}. $$
For an element $v \in V$ we denote its mean value by $\overline{v} = \langle v , {\bf 1} \rangle . $
This way the notation is also defined for duals of Sobolev spaces such as the dual space $\Hdual(\Omega)$ of $H^1(\Omega)$. We denote by $K=(-\Delta)^{-1}: \Hdualm(\Omega) \rightarrow H^1_\diamond(\Omega)$ the solution operator of the Neumann problem for the Poisson equation, i.e. $f \mapsto w$, where
\begin{equation*}
\int_\Omega  \nabla w \cdot\nabla \varphi ~dx = \langle f,\varphi\rangle, \qquad \forall~ \varphi \in H^1_\diamond(\Omega).
\end{equation*} 
We assume that $\partial \Omega$ is sufficiently regular so that $K$ maps continuously $L^2_\diamond(\Omega)$ into $H_\diamond^2(\Omega)$. We still denote by $K$ its extension from $D(-\Delta)_\diamond'$ to $L^2_\diamond(\Omega)$, where $D(-\Delta):=\{\varphi \in H^2(\Omega): \partial_\nu \varphi =0 \text{ on } \partial \Omega\}$ is the domain of the Laplace operator with homogeneous Neumann boundary conditions.

\begin{definition}[(Very) weak solution] \label{def:weaksolution}
Let $T>0$ and let
\begin{align*} 
u &\in L^\infty(0,T;L^1(\Omega)) \cap L^\infty(0,T;\Hdual(\Omega)), \\ 
v &\in L^\infty(0,T;H^1(\Omega)) \cap H^1(0,T;\Hdual(\Omega)) 
\end{align*}
We call $(u,v)$ a weak solution of System \eqref{sys:localsensing} on $(0,T)$ if $e^{-v} u \in L^2( (0,T)\times\Omega)$,
\begin{align}
K \partial_t u = - e^{-v} u + \overline{e^{-v} u} \qquad & \text{in } L^2( (0,T)\times\Omega), \label{eq:weaku}\\
\langle \partial_t v, \varphi \rangle + \int_\Omega \left( \epsilon \nabla v \cdot \nabla \varphi + \beta v \varphi \right) ~dx = 
\langle u, \varphi \rangle, \qquad & \forall \varphi \in H^1(\Omega), \, \text{a.e. in } (0,T), \label{eq:weakv}
\end{align}
and, setting $m:=\overline{u_0}$,
\begin{equation} \label{eq:weakinit}
\begin{split}
K (u(0,\cdot)-m) = K (u_0-m)   \qquad & \text{in } L^2(\Omega),\\
v(0,\cdot)=v_0   \qquad & \text{in } \Hdual(\Omega).
\end{split}
\end{equation}
\end{definition}

Notice that we interpret Equation~\eqref{PLe1} in a very weak sense, reminiscent to the treatment of the porous medium or similar nonlinear diffusion equations (cf., e.g., \cite{vazquez2007porous}). The initial condition \eqref{eq:weakinit} makes sense since $\partial_t K (u-m) = K \partial_t u \in L^2( (0,T)\times\Omega)$ by \eqref{eq:weaku} and $K(u-m) \in L^\infty(0,T;H^1(\Omega))$.
\begin{definition}[Weak-strong solution]
Let $T>0$. A weak-strong solution of System \eqref{sys:localsensing} on $(0,T)$ is a weak solution $(u,v)$ which satisfies additionally
\begin{equation*}
\partial_t v = \epsilon \Delta v - \beta v + u \qquad \text{in } L^2((0,T) \times \Omega). 
\end{equation*}
For $u$ and $v$ defined on $(0,+\infty)\times\Omega$, we call $(u,v)$ a global weak-strong solution of System \eqref{sys:localsensing} if it is a weak-strong solution on $(0,T)$ for all times $T>0$.
\end{definition} 

Our main result of existence of global solution is the following.
\begin{thm}[Existence of global solution] \label{thm:globalexistence}
Let $\epsilon > 0$, $\beta \geq 0$, and let $\Omega \subset \R^d$ ($d\ge 1$) be a smooth bounded domain of measure one. Assume that the initial conditions satisfy the following: 
\begin{equation}\label{hypidu}
\begin{split}
& u_0 \in L^1(\Omega) \cap \Hdual(\Omega), \quad u_0 \geq 0,  \\
& m := \int_\Omega u_0~dx, \quad \int_\Omega u_0 \log \frac{u_0}{m}~dx < \infty,
\end{split}
\end{equation}
and
\begin{equation}
v_0 \in H^1(\Omega), \quad v_0 \geq 0. \label{hypidv}
\end{equation}

Then there exist two nonnegative functions $u,\, v$ with
\begin{align*} u &\in L^\infty(0,T;L^1(\Omega)) \cap L^\infty(0,T;\Hdual(\Omega)), \\
v &\in L^\infty(0,T;H^1(\Omega)) \cap H^1(0,T;L^2(\Omega)),
\end{align*}
for all times $T>0$, such that $(u,v)$ is a global weak-strong solution of \eqref{sys:localsensing}. Furthermore it satisfies, for all $t>0$,
\begin{align*} 
& \int_\Omega u(t) ~dx = m,\\
\Vert u(t) -m \Vert_{\Hdual(\Omega)}^2 & + 2 \int_0^t \int_\Omega e^{-v} \, u^2~dx~ds \leq \Vert u_0 - m\Vert_{\Hdual(\Omega)}^2 + 2\,m^2t
,
\end{align*}
and
\begin{align*}
\int_\Omega ( \frac{\epsilon}4  \vert \nabla v \vert^2 + \frac{\beta}4 \vert  v \vert^2 + u \log u - u +1)~dx
 & + 4 \int_0^t \int_\Omega |\nabla \sqrt{e^{- v}u}|^2~dx~ds\\
 +  \int_0^t \int_\Omega |\epsilon \Delta v + u - \beta v|^2~dx~ds & \leq {\cal E}(u_0,v_0) + \frac{1}{\epsilon}\left({\Vert u_0 - m\Vert_{\Hdual(\Omega)}^2 + 2\,m^2t}\right)+ \frac{m^2}{\beta}.
\end{align*}
\end{thm}

\begin{rem}[Regularity of the domain]
The regularity required for the boundary of the domain is the minimal regularity such that $(I-\Delta)^{-1}$ (the self-adjoint solution operator for the homogeneous Neumann problem) maps $H^m(\Omega)$ into $H^{m+2}(\Omega)$ for all $m\in \{0,\dots,2d-1\}$. This property is used in the approximation procedure in Section~\ref{sec:globalexistence}. 
\end{rem}
\begin{rem}[The parabolic-elliptic problem]\label{rem:parell}
For $\epsilon > 0$, $\beta \geq 0$, $\Omega \subset \R^d$ ($d \ge 1$) a smooth bounded domain of measure one, and an initial condition $u_0$ satisfying \eqref{hypidu}, we consider the parabolic-elliptic problem
\begin{equation}\label{parell}
\begin{split}
\partial_t u &= \Delta (e^{-v} u), \qquad \qquad \text{in } (0,\infty)\times \Omega,\\
0 &= \epsilon \Delta v - \beta v + u, \qquad \text{in } (0,\infty)\times \Omega.
\end{split}
\end{equation}
We expect that our methods of proof of existence hold for this problem with minor adaptations, so that there exists a global weak solution of \eqref{parell} with no-flux boundary conditions and initial data $u_0$. We give in Section~\ref{sec:parell} some elements in this direction.
\end{rem}

Having verified the global existence of solutions it is natural to consider the question of uniqueness. Indeed we are able to prove uniqueness by the following result, with the restriction that we need to assume minimal additional regularity of $u$. 
\begin{thm}[Uniqueness of weak-strong solution]\label{thm:uniqueness}
Let $\epsilon > 0$, $\beta \geq 0$, and let $\Omega \subset \R^d$ ($d\ge 1$) be a smooth bounded domain of measure one. There is at most one weak-strong solution $(u,v)$ to \eqref{sys:localsensing} such that
\begin{equation*}
u\in L^\infty(0,T;L^q(\Omega)) \;\text{ for all }\; T>0\ , 
\end{equation*}
with $q>2$ when $d\in \{1,2\}$ and $q=d$ when $d\ge 3$.
\end{thm}

Due to the additional regularity needed for the uniqueness proof it is natural to investigate additional regularity, possibly with further conditions on the initial values. The results we can obtain in this case depend on the dimension however, so we will distinguish different cases. We start with spatial dimension one for completeness and then discuss the two-dimensional case, which appears to be the most interesting for applications and the expected mass-dependent blow-up.

\begin{prop}[Regularity, dimension 1]\label{reg_dim1}
Let $d=1$ and, in addition to assumptions of Theorem~\ref{thm:globalexistence}, $u_0 \in L^r(\Omega)$ for $r \geq 1$. Then, for any $T>0$, $u\in L^\infty(0,T;L^r(\Omega))$ and $v\in L^\infty((0,T)\times \Omega)$. In particular, for $u_0 \in L^r(\Omega)$ with $r > 2$, the solution is unique.
\end{prop}

In the following we consider the case $d=2$, the main result we prove is again a uniform bound on $u$ and $v$ for finite time, given the initial values are bounded:

\begin{prop}[Regularity, dimension 2]\label{P.PLr0}Let $d=2$ and, in addition to assumptions of Theorem \ref{thm:globalexistence}, $u_0 \in L^\infty(\Omega)$, $v_0 \in L^\infty(\Omega)$. Then, for any $T>0$, $u\in L^\infty((0,T)\times \Omega)$ and $v\in L^\infty((0,T)\times \Omega)$. As a consequence the solution is unique.
\end{prop}

A striking consequence of this result, compared to the minimal Keller-Segel model, is the absence of finite-time blow-up (in $L^\infty$) in dimension~2. Indeed, Proposition~\ref{P.PLr0} rules out finite time blow-up in $L^\infty$ for arbitrary initial data and extends \cite[Theorem~1.1]{jin2019keller}, where the existence of global solutions is established only when $m\in (0,4\pi\epsilon)$. Under this additional condition, it is also shown in \cite[Theorem~1.1]{jin2019keller} that $(u,v)\in L^\infty((0,\infty)\times \Omega; \mathbb{R}^2)$, see also \cite[Theorem~2.3]{FujieJiang20}. Also, combining  Proposition~\ref{P.PLr0} and \cite[Theorem~1.1]{jin2019keller} guarantees that there are initial data with $m>4\pi\epsilon$ for which the corresponding solution is global and unbounded, see also \cite[Theorem~2.3]{FujieJiang20}. Finally, Proposition~\ref{P.PLr0} also provides an alternative proof of \cite[Theorem~2.1]{FujieJiang20} for the specific choice $\gamma(s)=e^{-s}$, $s\ge 0$.

\subsection{Litterature on System \eqref{sys:localsensing}} The mathematical analysis of System \eqref{PLe1}--\eqref{PLe2} has been tackled in some very recent works. In \cite{TaoWinkler_M3AS} the authors choose the motility function $\gamma$ to be smooth and both upper and lower bounded by some positive constants. Under these assumptions, they obtain weak solutions in all spatial dimensions $d\ge2$, and, in dimension $d=2$, assuming furthermore that the domain $\Omega$ is convex, they obtain smooth solutions (and also in dimension $d\ge3$ under an appropriate smallness assumption on the initial data). In \cite{DKTY19}, the motility $\gamma$ is taken slowly algebraically decreasing (that is, with a constraint on the algebraic rate which depends on the spatial dimension). Weak solutions are obtained in spatial dimensions $d=1,2,3$ (actually smooth solutions in dimension $d=1$). The long-time behaviour is also explored through a numerical and linear stability analysis. We also mention the work \cite{YoonKim2017} where smooth solutions are obtained in the two-dimensional case for a motility $\gamma$ taken sufficiently small. Both works \cite{TaoWinkler_M3AS,DKTY19} rely on duality and energy techniques (but not on an entropy in the strict sense of the term). Shortly afterwards, in \cite{FujieJiang20}, smooth solutions are obtained in dimension $d=2$ for a very generic motility (positive and decreasing to zero). Furthermore, the authors obtain uniformly bounded smooth solutions if the motility $\gamma(v)$ decays at most algebraically in dimension $d=2$, and at most in $1/v$ in dimension $d=3$. Their method relies on the introduction of an adequate auxiliary elliptic problem that enjoys a comparison principle.

When it comes to the case where the motility has precisely the exponential decay \eqref{PLe3}, it has been studied specifically in two-dimensional domains in \cite{jin2019keller,FujieJiang20}. Both works focus on smooth solutions and the question of the possibility of blow-up. They show the existence of a critical mass such that, in the subcritical case, there exist uniformly bounded smooth solutions, and convergence to a single stationary solution is further shown in \cite{FujieJiang20}. In the supercritical case, they obtain the blow-up in $L^\infty$, which is furthermore proved to happen only in infinite time in \cite{FujieJiang20}.

Compared to these results, one added value of the present work is first that we obtain global solutions in any dimension $d\ge 1$. Furthermore, our solutions are weak (more precisely weak-strong) solutions, allowing to consider more general initial data. Note that one new difficulty that we encounter when dealing with weak solutions is the construction of an adequate approximation procedure. A crucial step is to define an appropriate approximate system that conserves both the entropy and the duality structures, which are of very different nature. Note that we can connect our weak solutions with the strong solutions obtained in \cite{jin2019keller,FujieJiang20} thanks to our regularity and uniqueness results.

Finally, we refer to \cite{FujieJiang20qs,AhnYoon19qs} for studies of the quasi-stationary model, where the second equation is replaced by its elliptic counterpart.

\subsection{Organization of the paper}

In Section~\ref{sec:structure} we discuss the basic properties of the model and present some \emph{a priori} estimates, which we make rigorous in Section~\ref{sec:globalexistence}. The latter is devoted to the global existence of weak solutions, which we derive by an approximation argument that preserves the generalized gradient structure and only modifies the cross-term in the entropy. In Section~\ref{sec:uniqueness} we investigate further regularity and uniqueness. Finally we conclude with some remarks and outlooks in Section~\ref{sec:conclusion}.

\vspace{2cm}

\section{Structure of the model and \emph{a priori} estimates} \label{sec:structure}

In the following we discuss the mathematical structure of the model and formally derive \emph{a priori} estimates, which will be made rigorous in existence proofs in Section \ref{sec:globalexistence}.

Since we are looking for solutions having a finite entropy ${\cal E}$ as defined in \eqref{def:entropy}, we restrict ourselves to meaningful initial values that lead to ${\cal E}(u_0,v_0) < \infty$, namely we assume that $u_0$ and $v_0$ satisfy \eqref{hypidu} and \eqref{hypidv}, respectively.

\subsection{Basic \emph{a priori} estimates} 

As mentioned in the introduction the system has some formal gradient structure and natural mass 
preservation, so we can formally derive estimates from these structures. We will assume sufficient regularity to justify the computations in this section and make the results rigorous later.

\paragraph{Mass.}
Let $t>0$. First of all, the system preserves nonnegativity and the mean value of $u$, thus we obtain
$$ \Vert u(t) \Vert_{L^1(\Omega)} = \bar{u}(t) = m, $$
and 
$$ \frac{d}{dt} \left( \bar{v}(t) - \frac{m}{\beta}\right)= m - \beta\bar{v}(t) , $$
which implies 
$$ \bar{v}(t) = \frac{m}{\beta} + e^{-\beta t} \lk \overline{v_0} - \frac{m}{\beta}\rk. $$
In particular since $v$ is also nonnegative
$$ \Vert v(t) \Vert_{L^1(\Omega)} \leq \max\{ \frac{m}{\beta}, \Vert v_0 \Vert_{L^1(\Omega)} \}. $$

\paragraph{Entropy dissipation.}
We recall the definition \eqref{def:entropy} of the entropy:
\begin{equation*}
{\cal E}(u,v) =  \int ( \frac{\epsilon}2  \vert \nabla v \vert^2  + \frac{\beta}2 \vert  v \vert^2 - u v + u \log u - u +1)~dx,
\end{equation*}   
and emphasize that it is in general not bounded from below due to the negative and quadratic term $-uv$, see \cite{BBTW2015,horstmann20031970,MR2448428,stevens1997aggregation,Perthame}. From the entropy dissipation we obtain the following \emph{a priori} identity
$$ \frac{d}{dt}{\cal E}(u(t),v(t)) = - 4 \int_\Omega |\nabla \sqrt{e^{-v}u}|^2~dx -  \int_\Omega |\epsilon \Delta v + u - \beta v|^2~dx.   $$
Note that, from the entropy dissipation, we expect that $u e^{-v}$ and $\epsilon \Delta v + u - \beta v$ have increased regularity, provided the entropy is bounded from below. This gives a particular motivation for the definition of weak-strong solutions given above. 

\paragraph{Duality estimate.}
Finally, we have an estimate in $\Hdual$, which follows by testing Equation~\eqref{eq:weaku} with $u- m$. We see that 
$$ \frac{1}2 \frac{d}{dt} \langle u- m,K(u-m)  \rangle = - \langle u- m, e^{-v} u - \overline{e^{-v} u} \rangle = - \int_\Omega e^{-v} u^2 ~dx + m~ \overline{e^{-v} u} .$$
Owing to the nonnegativity of $u$ and $v$,
$$   \overline{e^{-v} u} = \int_\Omega e^{-v} u ~dx \leq \int_\Omega  u ~dx = m , $$ 
so that, for all $T>0$,
\begin{equation} \label{eq:dualityestimate}
\frac{1}2 \sup_{[0,T]}  \Vert u - m \Vert_{\Hdual(\Omega)}^2 + \int_0^T \int_\Omega e^{-v} u^2 ~dx~dt \leq  \frac{1}2 \Vert u_0 - m \Vert_{\Hdual(\Omega)}^2 + m^2 T. 
\end{equation}

\subsection{Entropy on $\Hdual$-bounded sets}\label{sec:entropy_on_Hdual}

A key property of the entropy is that, due to the nonlinear negative term $-uv$, it is in general not bounded from below, except when $d=1$ or $d=2$ in the subcritical case. This feature leads to the occurrence of finite time blow-up for the  minimal Keller-Segel model in the supercritical case when $d=2$ and in higher spatial dimensions $d\ge 3$. Since the model \eqref{PLe1}, \eqref{PLe2} has the same entropy functional, a similar behaviour is to be expected. However, the key difference is that the changed dynamics results in a bound on the $\Hdual$-norm of $u$ for finite time, and the entropy is bounded below if there is such an additional $\Hdual$-bound.

Indeed, suppose that two nonnegative functions $u$ and $v$ defined on $\Omega$ satisfy
$$\Vert u - m \Vert_{\Hdual(\Omega)} \leq C,$$
% \qquad \text{and} \qquad \Vert v \Vert_{L^1(\Omega)} \leq C'$$
for some positive constant $C$. Then we can estimate
\begin{align*} 
- \int_\Omega u v~dx &=- \int_\Omega (u-m) v~dx - \int_\Omega m v~dx = \int_\Omega v\, \Delta K(u-m) ~dx - \int_\Omega m v~dx \\
& = - \int_\Omega \nabla K(u-m) \cdot \nabla v ~dx - \int_\Omega m v~dx \\
& \ge - \frac{C^2}{\epsilon} - \frac{\epsilon}4 \int_\Omega |\nabla v|^2 ~dx -  \frac{m^2}\beta - \frac{\beta}4 \int_\Omega |v|^2~dx.
\end{align*}
As a consequence, we can estimate the entropy from below by
\begin{align*} 
{\cal E}(u,v) &=  \int_\Omega \left( \frac{\epsilon}2  \vert \nabla v \vert^2  +  \frac{\beta}2 \vert  v \vert^2 - u v + u \log u - u +1 \right)~dx \\
&\geq - \frac{C^2}{\epsilon}- \frac{m^2}{\beta} + \int_\Omega \left( \frac{\epsilon}4  \vert \nabla v \vert^2 + \frac{\beta}4 \vert  v \vert^2 + u \log u - u +1 \right)~dx,
\end{align*}
where the last integral is clearly nonnegative. Thus, boundedness of the entropy functional implies an upper bound on the $H^1$-norm of $v$ and on the logarithmic entropy of $u$ for every finite time. This holds true even in the supercritical case, hence if blow-up occurs, then it has to be postponed to the infinite time limit.

\section{Global existence of weak solutions} \label{sec:globalexistence}

In the following we prove the existence of global solutions (Theorem~\ref{thm:globalexistence}) by an appropriate approximation procedure. Since we have several \emph{a priori} estimates based on the entropy dissipation and the structure of the system ($\Hdual$ estimates in the first equation), it is important to keep these basic structures in the approximation. To this end, for $\nu\in (0,1)$, we introduce the system
\begin{subequations}\label{sys:approx}
\begin{align}
\partial_t u &= \Delta (e^{-L_\nu v} u), & \text{in } (0,\infty)\times \Omega, \label{approx1} \\
\partial_t v &= \epsilon \Delta v - \beta v + L_\nu   u, & \text{in } (0,\infty)\times \Omega, \label{approx2}
\end{align}
together with no-flux boundary conditions
\begin{align}
\pa_\nu [e^{-L_\nu v} u] = \pa_\nu v =0, \qquad \text{ on } (0,\infty)\times \pa \Omega, \label{approxNeumann}
\end{align}
\end{subequations}
and initial data $(u_0,v_0)$. Here, $L_\nu=\Lambda_\nu^d$ (with $d$ the spatial dimension) and $\Lambda_\nu= (I-\nu \Delta)^{-1}$ is the positivity-preserving self-adjoint (positive definite) solution operator for the homogeneous Neumann problem with small diffusion coefficient $\nu$, defined from $\Hdual(\Omega)$ onto $H^1(\Omega)$. Note that with our regularity conditions on the domain $\Omega$, $\Lambda_\nu$ also maps continuously $H^m(\Omega)$ into $H^{m+2}(\Omega)$ for all $m\in \{0,\dots,2d-1\}$. As a consequence, $L_\nu$ maps continuously $\Hdual(\Omega)$ into $H^{2d-1}(\Omega) \hookrightarrow C(\overline\Omega)$. In the limit $\nu \rightarrow 0$, $L_\nu\to I$ and we expect to recover the original system.

The approximate system \eqref{sys:approx} preserves the gradient structure but only mollifies the entropy to
\begin{equation}\label{def:E_nu}
{\cal E}_\nu(u,v) =  \int_\Omega ( \frac{\epsilon}2  \vert \nabla v \vert^2  +  \frac{\beta}2 \vert  v \vert^2 -  (L_\nu u) v + u \log u - u +1)~dx.
\end{equation} 
Therefore we can derive rigorously at the level of the approximate system some uniform-in-$\nu$ estimates which are reminiscent of the previously described \emph{a priori} estimates for the original system.
The basic idea of the proof is to first establish the existence of solutions for \eqref{sys:approx} via a fixed-point technique and then use uniform-in-$\nu$ estimates to pass to the limit $\nu \rightarrow 0$. 

\subsection{Existence for the approximate system}

Our first objective is to establish the following existence result for the approximate problem~\ref{sys:approx}:
\begin{thm}\label{thm:ptfx}
Let $T>0$. Let $u_0 \in L^2(\Omega)$, $v_0 \in H^1(\Omega)$ both be nonnegative with $m:=\overline {u_0} >0$. Then, there exists a weak solution 
\begin{align*}
u_\nu &\in L^2(0,T;H^1(\Omega)) \cap L^\infty(0,T;L^2(\Omega)) \cap H^1(0,T;\Hdual(\Omega)), \\
v_\nu &\in C([0,T];H^1(\Omega)) \cap C^1([0,T];\Hdual(\Omega))
\end{align*}
of \eqref{sys:approx} with initial condition $u_0,v_0$. Moreover, $u$ and $v$ are nonnegative almost everywhere on $(0,T)\times\Omega$ and, for all $t\in [0,T]$,
\begin{equation*}
\overline {u_\nu} (t) = m,  \qquad  \Vert u_\nu(t) -m \Vert_{\Hdual(\Omega)} \leq \sqrt{\Vert u_0 - m\Vert_{\Hdual(\Omega)}^2 + 2\,m^2T}.
\end{equation*}
\end{thm}

The proof is based on an application of Schauder's fixed point theorem and a comparison principle to establish the additional nonnegativity. We consider the fixed-point map 
\begin{equation*}
{\cal F} = {\cal I} \circ {\cal G},
\end{equation*} 
with 
\begin{align*}
{\cal G}:   C([0,T];\Hdual(\Omega)) &\rightarrow L^2(0,T;H^1(\Omega))\cap L^\infty(0,T;L^2(\Omega)) \cap H^1(0,T;\Hdual(\Omega)) \\ 
\tilde u  &\mapsto u,
\end{align*} 
where $(u,v)$ solves
\begin{align}
\partial_t u &= \Delta (e^{-L_\nu {v}} u ) = \nabla \cdot(  e^{-L_\nu {v}}  \nabla u - u e^{-L_\nu {v}} \nabla L_\nu v), \label{eq:ptfx1} \\
\partial_t v &= \epsilon \Delta v - \beta v + L_\nu  {\tilde u} \label{eq:ptfx2}
\end{align}
with no-flux boundary conditions and initial condition $(u_0,v_0)$. The operator ${\cal I}$ is simply the embedding operator
 \begin{equation*}
{\cal I}: L^2(0,T;H^1(\Omega))\cap  L^\infty(0,T;L^2(\Omega)) \cap H^1(0,T;\Hdual(\Omega)) \rightarrow C([0,T];\Hdual(\Omega)).
\end{equation*}
  
\begin{lemma}\label{lem:calG}
Let $u_0 \in L^2(\Omega)$ be nonnegative, $v_0 \in H^1(\Omega)$. Then 
the  operator ${\cal G}$ is well-defined and continuous. Moreover, the solution $u$ is nonnegative almost everywhere in $(0,T)\times\Omega$ and the same holds for $v$ if $v_0$ and $\tilde u$ are nonnegative. 
\end{lemma}
\begin{proof}
Let $\tilde u \in C([0,T];\Hdual(\Omega))$. First of all, $L_\nu  {\tilde u} \in C([0,T];H^{2d-1}(\Omega))$. Since $v_0\in H^1(\Omega)$, classical solvability of Eq. \eqref{eq:ptfx2} yields a unique solution $v$ in $C([0,T];H^1(\Omega))\cap H^1(0,T;L^2(\Omega))$. By \eqref{eq:ptfx2}, $\partial_t v \in C([0,T];\Hdual(\Omega))$, so that
$$v \in C([0,T];H^1(\Omega)) \cap C^1([0,T];\Hdual(\Omega)).$$
Moreover, if $\tilde u$ and $v_0$ are nonnegative, then so are $L_\nu \tilde u$ and $v$ by the comparison principle.

Next, $L_\nu v \in C([0,T];H^{1+2d}(\Omega))$ and thus $L_\nu v \in C([0,T];W^{1,\infty}(\Omega))$. Consequently, $e^{-L_\nu v}$ is an element of $L^\infty((0,T)\times\Omega)$ which is positive and bounded away from zero, while $\nabla e^{-L_\nu v}=- e^{-L_\nu v} \nabla L_\nu v$ is also an element of $L^\infty((0,T)\times\Omega)$. Thanks to these properties and the nonnegativity of $u_0$, the existence and uniqueness of a nonnegative weak solution
$$u \in L^2(0,T;H^1(\Omega))\cap L^\infty(0,T;L^2(\Omega)) \cap H^1(0,T;\Hdual(\Omega))$$
of Eq. \eqref{eq:ptfx1} follow from \cite{aronson1968non}.

The continuous dependence of $u$ on $L_\nu v $ and hence on $v$ follows from a standard difference technique using the bounds of solutions in terms of the coefficients in \cite{aronson1968non}.
\end{proof}

Using the compact embedding of $L^2(\Omega)$ in $\Hdual(\Omega)$, an application of a version of the Aubin-Lions-Simon lemma (cf. \cite[Corollary~4]{simon1986compact}) further yields:
\begin{lemma}\label{lem:calI}
The embedding operator ${\cal I}$ is compact. 
\end{lemma}
 
As a consequence of the two above lemmas, we see that the operator ${\cal F}$ is compact and continuous from $C([0,T];\Hdual(\Omega))$ into itself. In order to apply Schauder's fixed point theorem it remains to verify that it maps a closed bounded convex set into itself, which is the purpose of the following lemma. 
\begin{lemma} \label{selfmappinglemma}
Let $u_0 \in L^2(\Omega)$ be nonnegative with $m=\overline {u_0}$ and let $v_0 \in H^1(\Omega)$ be nonnegative.
For $R\ge m>0$, we introduce the bounded, closed, and convex sets
$$ {\cal M}_R: = \{ u \in C([0,T];\Hdual(\Omega))~|~ \|u(t)-m\|_{\Hdual(\Omega)}\le R \text{ and } \overline u(t) = m , \text{ for }t \in [0,T] \} $$
and
$$ {\cal M}^+_R: = \{ u \in  {\cal M}_R~|~ u(t) \ge 0 \text { in } {\cal D}'(\Omega) , \text{ for }t \in [0,T] \}. $$

Then for $R\geq m$ sufficiently large, the operator ${\cal F}$ maps ${\cal M}^+_R$ into itself.
\end{lemma}

\begin{proof}
Let $\tilde u \in {\cal M}^+_R$ for some $R\ge m$ to be specified later. We show that $u={\cal F}(\tilde u)$ belongs to ${\cal M}^+_R$. Since the solution $u$ of \eqref{eq:ptfx1} is nonnegative, and since Equation~\eqref{eq:ptfx1} is of divergence form with no-flux boundary conditions, we readily obtain
$$ \overline{u}(t) = m, \text{ for all }t \in [0,T].   $$
Owing to the regularity of $u$ and $v$ established in the proof of Lemma~\ref{lem:calG}, Equation~\eqref{eq:ptfx1} may be equivalently written
$$ K \partial_t u = - e^{- L_\nu v} u + \overline{e^{- L_\nu v} u},$$
so that, using the nonnegativity of $L_\nu v$, we may proceed as in the derivation of the estimate \eqref{eq:dualityestimate} and obtain
$$ \Vert u -m \Vert_{C([0,T];\Hdual(\Omega))} \leq \sqrt{\Vert u_0 - m\Vert_{\Hdual(\Omega)}^2 + 2\, m^2T }. $$
Thus, $u$ lies in ${\cal M}^+_R$ provided 
$$ R \geq \sqrt{\Vert u_0 - m\Vert_{\Hdual(\Omega)}^2 + 2\, m^2T }, $$
and the proof is complete.
\end{proof}

We are now able to conclude the proof of the existence theorem:

\begin{proof}[Proof of Theorem~\ref{thm:ptfx}]
The proof directly follows from the application of Schauder's fixed point theorem to the operator $\cal F$ which is possible thanks to Lemmas \ref{lem:calG}, \ref{lem:calI} and \ref{selfmappinglemma}.
\end{proof}

\subsection{Uniform estimates for the approximate system}

In the following we collect the uniform-in-$\nu$ estimates we can get for the approximate solutions $(u_\nu,v_\nu)$ of \eqref{sys:approx}. They will be helpful to pass to the limit $\nu\rightarrow 0$ in the next subsection.

\begin{prop}\label{prop:unif_est}
Let $T>0$. Let $u_0 \in L^2(\Omega)$, $v_0 \in H^1(\Omega)$ both be nonnegative, and let $(u_\nu,v_\nu)$ be a weak solution of \eqref{sys:approx} given by Theorem \ref{thm:ptfx}. Then, $(u_\nu,v_\nu)$ satisfies the following estimates:
\begin{align} \label{apriori1}
\Vert u_\nu(t) \Vert_{L^1(\Omega)}  & = m, \qquad t\in [0,T],\\
\label{apriori2}
 \Vert u_\nu \Vert_{L^\infty(0,T;\Hdual(\Omega))} & \leq C(T),\\
\label{apriori3}
\sup_{t \in [0,T]}\int u_\nu(t) \log \frac{u_\nu(t)}m ~dx & \leq C(T),\\
\Vert v_\nu \Vert_{L^\infty(0,T;H^1(\Omega))} & \leq C(T), \label{apriori3.5}\\
\label{apriori4}
\Vert -\epsilon \Delta v_\nu + \beta v_\nu - L_\nu u_\nu \Vert_{L^2((0,T)\times\Omega)} & \leq C(T),\\
\label{apriori5}
\Vert \partial_t v_\nu \Vert_{L^2((0,T)\times\Omega)} & \leq C(T),\\
 \label{apriori6b}
 \Vert e^{-L_\nu v_\nu} u_\nu \Vert_{L^2((0,T)\times\Omega)}& \leq C(T), \\
 \label{apriori6c}
\Vert K \partial_t u_\nu \Vert_{L^2((0,T)\times\Omega)} & \leq C(T),
\end{align} 
for some constant $C(T)$ depending only on the initial data (more precisely on $m$, $\Vert u_0\Vert_{\Hdual(\Omega)}$, $\int_\Omega u_0 \log (u_0/m)~dx$ and $\Vert v_0\Vert_{H^1(\Omega)}$) and time $T$, but independent of $\nu$.
\end{prop}

From Theorem~\ref{thm:ptfx} we already know that the solutions we obtained satisfy \eqref{apriori1} and \eqref{apriori2}. We furthermore have the following dual estimate.

\begin{lemma}
For almost every $t \in (0,T)$, we have
\begin{equation}\label{unif_duality}
\Vert u_\nu(t) -m \Vert_{\Hdual(\Omega)}^2 + 2 \int_0^t \int_\Omega e^{-L_\nu v_\nu} \, u_\nu^2~dx~ds \leq \Vert u_0 - m\Vert_{\Hdual(\Omega)}^2 + 2\,m^2t.
\end{equation}
\end{lemma} 

\begin{proof}
Thanks to the nonnegativity of $L_\nu v_\nu$, we can proceed as in \eqref{eq:dualityestimate} and we obtain the result.
\end{proof}

To obtain the other bounds we rely on the entropy dissipation.
\begin{lemma}
Recall the definition of ${\cal E}_\nu$ in \eqref{def:E_nu}. Then for almost every $t \in (0,T)$, we have the estimate
\begin{equation}\label{unif_entropy}
\begin{split}
{\cal E}_\nu(u_\nu(t),v_\nu(t)) & + 4 \int_0^t \int_\Omega |\nabla \sqrt{e^{-L_\nu v_\nu}u_\nu}|^2~dx~ds \\
& +  \int_0^t \int_\Omega |\epsilon \Delta v_\nu + L_\nu u_\nu - \beta v_\nu|^2~dx~ds \leq {\cal E}_\nu(u_0,v_0).
\end{split}
\end{equation}
\end{lemma} 

\begin{proof}
First of all, we use the properties $v_0 \in H^1(\Omega)$ and $L_\nu u_\nu \in L^2((0,T)\times\Omega)$ to show that $\partial_t v_\nu$ and $\Delta v_\nu$ both lie in $L^2((0,T)\times\Omega)$. 
Now we can employ a standard proof to derive entropy dissipation estimates: for $\delta>0$, using 
$$ \log (\delta + u_\nu) - L_\nu v_\nu =  \log( (\delta + u_\nu )e^{-L_\nu v_\nu}) \in L^2(0,T;H^1(\Omega))$$ 
as test function in \eqref{approx1} and
$$ -\epsilon \Delta v_\nu - L_\nu u_\nu + \beta v_\nu \in L^2((0,T)\times\Omega) $$
in \eqref{approx2},
then summing the two equalities and integrating in time we get
\begin{equation}
\begin{split}
\int_\Omega \left( (\delta +u_\nu)\,(\log(\delta+u_\nu)-1)+1 + \frac \epsilon 2 |\nabla v_\nu|^2 + \frac \beta 2 |v_\nu|^2 - v_\nu \,L_\nu u_\nu \right) ~dx\\
+ \int_0^t \int_\Omega \frac{\nabla (e^{-L_\nu v_\nu}u_\nu)\cdot \nabla (e^{-L_\nu v_\nu}(\delta+u_\nu))}{e^{-L_\nu v_\nu}(\delta+u_\nu)}~dx~ds
+  \int_0^t \int_\Omega |\epsilon \Delta v_\nu + L_\nu u_\nu - \beta v_\nu|^2~dx~ds\\
= \int_\Omega \left( (\delta +u_0)\,(\log(\delta+u_0)-1)+1 + \frac \epsilon 2 |\nabla v_0|^2 + \frac \beta 2 |v_0|^2  - v_0 L_\nu u_0 \right) ~dx.
\end{split}\label{standard}
\end{equation}
Now, 
\begin{align*}
I_{\nu,\delta} & = \int_0^t \int_\Omega \frac{\nabla (e^{-L_\nu v_\nu}u_\nu)\cdot \nabla (e^{-L_\nu v_\nu}(\delta+u_\nu))}{e^{-L_\nu v_\nu}(\delta+u_\nu)}~dx~ds \\
& = 4 \int_0^t \int_\Omega \left| \nabla\sqrt{e^{-L_\nu v_\nu}(\delta+u_\nu)} \right|^2~dx~ds \\
& \quad + \int_0^t \int_\Omega \frac{\delta}{\delta+u_\nu} \nabla L_\nu v_\nu \cdot \nabla (e^{-L_\nu v_\nu}(\delta+u_\nu))~dx~ds.
\end{align*}
On the one hand, weak lower semicontinuity ensures that
\begin{equation*}
\liminf_{\delta\to 0} 4 \int_0^t \int_\Omega \left| \nabla\sqrt{e^{-L_\nu v_\nu}(\delta+u_\nu)} \right|^2~dx~ds \ge 4 \int_0^t \int_\Omega \left| \nabla\sqrt{u_\nu e^{-L_\nu v_\nu}} \right|^2~dx~ds.
\end{equation*}
On the other hand, 
\begin{align*}
& \int_0^t \int_\Omega \frac{\delta}{\delta+u_\nu} \nabla L_\nu v_\nu \cdot \nabla (e^{-L_\nu v_\nu}(\delta+u_\nu))~dx~ds \\
& = \int_0^t \int_\Omega \frac{\delta}{\delta+u_\nu}e^{-L_\nu v_\nu} \nabla L_\nu v_\nu \cdot \nabla u_\nu~dx~ds - \delta \int_0^t \int_\Omega e^{-L_\nu v_\nu} | \nabla L_\nu v_\nu |^2~dx~ds \\
& = \int_0^t \int_\Omega \frac{\delta \mathbf{1}_{(0,\infty)}(u_\nu)}{\delta+u_\nu}e^{-L_\nu v_\nu} \nabla L_\nu v_\nu \cdot \nabla u_\nu~dx~ds - \delta \int_0^t \int_\Omega e^{-L_\nu v_\nu} | \nabla L_\nu v_\nu |^2~dx~ds.
\end{align*}
Since $0 \le \delta \mathbf{1}_{(0,\infty)}(u_\nu)/(\delta+u_\nu)\le 1$ and converges to zero as $\delta \to 0$, we deduce with the help of Lebesgue's dominated convergence theorem and the $L^2(0,T;H^1(\Omega))$-regularity of $u_\nu$ and $v_\nu$ that
\begin{equation*}
\lim_{\delta\to 0} \int_0^t \int_\Omega \frac{\delta}{\delta+u_\nu} \nabla L_\nu v_\nu \cdot \nabla (e^{-L_\nu v_\nu}(\delta+u_\nu))~dx~ds = 0.
\end{equation*}
Consequently,
\begin{equation*}
\liminf_{\delta\to 0} I_{\nu,\delta} \ge 4 \int_0^t \int_\Omega \left| \nabla\sqrt{u_\nu e^{-L_\nu v_\nu}} \right|^2~dx~ds,
\end{equation*}
and we may let $\delta\to 0$ in \eqref{standard} to complete the proof.
\end{proof}

We can now obtain the remaining estimates, thanks to the properties of the entropy on $\Hdual$-bounded sets.

\begin{proof}[Proof of Proposition~\ref{prop:unif_est}]
Reasoning as in Section~\ref{sec:entropy_on_Hdual}, we can use estimate \eqref{apriori2} to estimate the entropy from below by
\begin{align} \label{entropy_on_Hdual_nu}
&{\cal E_\nu}(u_\nu,v_\nu) \nonumber \\
&=  \int_\Omega ( \frac{\epsilon}2  \vert \nabla v_\nu \vert^2  +  \frac{\beta}2 \vert  v_\nu \vert^2 - u_\nu L_\nu v_\nu + u_\nu \log u_\nu - u_\nu +1)~dx \nonumber \\
&\geq - \frac{C(T)^2}{\epsilon}- \frac{m^2}{\beta} + \int_\Omega ( \frac{\epsilon}2  \vert \nabla v_\nu \vert^2  - \frac{\epsilon}4  \vert \nabla L_\nu v_\nu \vert^2+ \frac{\beta}2 \vert  v_\nu \vert^2 - \frac{\beta}4 \vert  L_\nu v_\nu \vert^2 + u_\nu \log u_\nu - u_\nu +1)~dx \nonumber \\
&\geq - \frac{C(T)^2}{\epsilon}- \frac{m^2}{\beta} + \int_\Omega ( \frac{\epsilon}4  \vert \nabla v_\nu \vert^2 + \frac{\beta}4 \vert  v_\nu \vert^2 + u_\nu \log u_\nu - u_\nu +1)~dx,
\end{align}
where the last inequality comes from the fact that the operator $L_\nu$ is a contraction in $L^2(\Omega)$. Reinserting this last estimate in the entropy dissipation \eqref{unif_entropy} and observing that
\begin{equation*}
\int_\Omega \left( u_\nu \log u_\nu - u_\nu +1 \right)~dx = \int_\Omega u_\nu \log \frac{u_\nu}{m}~dx + m\log m - m + 1 \ge \int_\Omega u_\nu \log \frac{u_\nu}{m}~dx, 
\end{equation*} 
we directly get estimates \eqref{apriori3}, \eqref{apriori3.5}, and \eqref{apriori4}.

Estimate \eqref{apriori5} follows directly from \eqref{apriori4}, using \eqref{approx2}. Using furthermore $0 \leq e^{-L_\nu v_\nu} \leq 1$ together with \eqref{unif_duality}, we get \eqref{apriori6b}. Rewriting Eq. \eqref{approx1} as
\begin{equation*}
K \partial_t u_{\nu} = - e^{-L_\nu v_{\nu}} u_{\nu} + \overline{e^{-L_\nu v_{\nu}} u_{\nu}},
\end{equation*}
we deduce \eqref{apriori6c} from \eqref{apriori1} and \eqref{apriori6b}.
\end{proof}

\subsection{Passage to the limit in the approximate system}
We are now ready to prove Theorem~\ref{thm:globalexistence} by passing to the limit as $\nu\to 0$ for solutions of \eqref{sys:approx}.

\begin{proof}[Proof of Theorem~\ref{thm:globalexistence}]
We first suppose that $u_0\in L^2(\Omega)$, so that we can use the approximation procedure previously built. We employ the estimates \eqref{apriori2}-\eqref{apriori6c} to extract a weakly-star convergent subsequence $(u_{\nu_n},v_{\nu_n})_{n \in \N}$ with the following properties:
\begin{align*}
u_{\nu_n}& \stackrel{*}{\rightharpoonup} u &&\text{in } L^\infty(0,T;\Hdual(\Omega)), \\
u_{\nu_n} & \rightharpoonup u && \text{in } L^1((0,T)\times \Omega), \\
K(u_{\nu_n}-m) & \rightharpoonup^* K(u-m) &&\text{in } L^\infty(0,T;H^1(\Omega)), \\
\partial_t K(u_{\nu_n}-m)&\rightharpoonup  \partial_t K(u-m) &&\text{in } L^2((0,T)\times\Omega)), \\
v_{\nu_n}& \stackrel{*}{\rightharpoonup} v &&\text{in } L^\infty(0,T;H^1(\Omega)), \\
\partial_t v_{\nu_n}&\rightharpoonup  \partial_t v &&\text{in }  L^2((0,T)\times\Omega), \\
e^{-L_{\nu_n} v_{\nu_n}} u_{\nu_n}&\rightharpoonup w &&\text{in } L^2((0,T)\times\Omega).
\end{align*}
Moreover, from the Aubin-Lions-Simon theorem (cf. \cite[Corollary~4]{simon1986compact}), we can immediately deduce that up to the extraction of a subsequence $(v_{\nu_n})_{n\in\N}$ and $(K(u_{\nu_n}-m)_{n\in\N})$ converge strongly to $v$ and $K(u-m)$ in $C([0,T];L^2(\Omega))$, respectively. The properties of $L_\nu$ also imply that $L_{\nu_n} v_{\nu_n} \rightarrow v$ in $L^2((0,T)\times\Omega)$ and $L_{\nu_n} u_{\nu_n} \rightharpoonup^* u$
in $L^\infty(0,T;\Hdual(\Omega))$.
Then we can pass to the limit in the second equation \eqref{approx2}, which is linear and also in the left-hand side of the first equation 
$$ K \partial_t u_{\nu_n} = - e^{-L_{\nu_n} v_{\nu_n}} u_{\nu_n} + \overline{e^{-L_{\nu_n} v_{\nu_n}} u_{\nu_n}},$$
which is obtained from \eqref{approx1}. Finally, we use that $(e^{-L_{\nu_n} v_{\nu_n}})_{n\in\N}$ converges pointwise towards $e^{-v}$ and is uniformly bounded (since $|e^{-L_{\nu_n} v_{\nu_n}}|\le 1$), and that $(u_{\nu_n})_{n\in\N}$ converges weakly in $L^1((0,T)\times\Omega)$ to pass in the weak limit in the product $e^{-L_{\nu_n} v_{\nu_n}} u_{\nu_n}$ and identify $w=e^{-v} u$ (see for example \cite[Proposition~2.61]{FonsecaLeoni}). This shows that $(u,v)$ is a weak-strong solution of \eqref{sys:localsensing} on $(0,T)$, and by Cantor's diagonal argument, we can build a global (in time) weak-strong solution of \eqref{sys:localsensing}, which completes the proof of Theorem~\ref{thm:globalexistence} for initial values $u_0 \in L^2(\Omega)$ and $v_0 \in H^1(\Omega)$. 

Now, if $u_0\notin L^2(\Omega)$, since estimates \eqref{apriori1}--\eqref{apriori6c} do not depend on $\Vert u_0 \Vert_{L^2(\Omega)}$, but only on $m$, $\int_\Omega u_0 \log(u_0/m) dx$ and $\Vert u_0 \Vert_{\Hdual(\Omega)}$, we can replace $u_0$ by a more regular ($L^2$) approximation $u_{0,\nu}$ in the approximation procedure and pass to the limit in the same way. This completes the proof of the existence of a global weak-strong solution of \eqref{sys:localsensing}.

Finally, we can pass to the limit in estimates \eqref{unif_duality}, \eqref{unif_entropy}, and \eqref{entropy_on_Hdual_nu} and obtain the estimates listed in Theorem~\ref{thm:globalexistence}.
\end{proof}

\subsection{The parabolic-elliptic problem}\label{sec:parell}

As announced in Remark~\ref{rem:parell}, we give here some elements for the proof of the global existence for the parabolic-elliptic problem
\eqref{parell}. We recall the problem
\begin{align}
\partial_t u &= \Delta (e^{-v} u), \qquad \text{in } (0,\infty)\times \Omega, \label{parell1}\\
0 &= \epsilon \Delta v - \beta v + u, \qquad \text{in } (0,\infty)\times \Omega, \label{parell2}
\end{align}
with no-flux boundary conditions and nonnegative initial condition $u_0$ satisfying \eqref{hypidu}. We have a similar gradient structure as in the fully parabolic case, but the second equation is even at equilibrium. In particular mass conservation, duality estimates, and entropy dissipation are preserved. The only difference is the bound on the time derivative of $v$, which needs to replaced by the time differentiated version of \eqref{parell2} 
$$ - \epsilon \Delta \partial_t v + \beta \partial_t v = \partial_t u, $$
which implies 
$$ ( I + \frac{\beta}\epsilon K) \partial_t v =  \frac{1}\epsilon K \partial_t u \in L^2((0,T)\times\Omega) ,$$
taking into account that $\bar{v}(t)=m/\beta$ for $t\ge 0$ by \eqref{parell2}. Hence, there is an \emph{a priori} bound on $\partial_t v$ that can be used with the Aubin-Lions-Simon theorem (cf. \cite[Corollary~4]{simon1986compact}) in order to establish strong compactness of $v$. Indeed, using this strategy, the existence of a solution to \eqref{parell} can be established with the same kind of proof as above based on the approximate problem
\begin{align*}
\partial_t u  &= \Delta (e^{-L_\nu v} u), \qquad \text{in } (0,\infty)\times \Omega, \\
0 &= \epsilon \Delta v - \beta v + L_\nu u,  \qquad \text{in } (0,\infty)\times \Omega,
\end{align*}
after passing to the limit $\nu \rightarrow 0$, using the mass, entropy and duality bounds in an appropriate way.
 
\section{Uniqueness and Regularity}\label{sec:uniqueness}
%%%%%%%%%%%%%%%%
%%%%%%%%%%%%%%%%
We now investigate uniqueness of the solutions (in any dimension $d$) and their regularity (in dimension 1 and 2).

\paragraph{Notation.} In this section we use the short notation $\| u \|_p$ for the norm $\| u \|_{L^p(\Omega)}$ for $p\in [1,\infty]$.

\subsection{Uniqueness}
In this section we prove the result of uniqueness stated in Theorem~\ref{thm:uniqueness}.

\begin{proof}[Proof of Theorem~\ref{thm:uniqueness}]
Let $(u_i,v_i)$, $i=1,2$, be two solutions of \eqref{sys:localsensing} with respective initial conditions $(u_{0,i},v_{0,i})$, $i=1,2$ such that
\begin{equation*}
\overline{u_{0,i}} = m\ , \qquad i=1,2,\,
\end{equation*}
 and assume that 
\begin{equation}
u_i \in L^\infty(0,T;L^q(\Omega))\ , \;\text{ for all }\; T>0\ , \qquad i=1,2.\label{PLu0}
\end{equation}
Then we know from nonnegativity and mass conservation that
\begin{equation*}
\|u_1(t)\|_1 = \| u_2(t)\|_1 = m\ , \qquad t\ge 0\ , 
\end{equation*}
so that $U:=u_1-u_2$ has zero average with respect to the space variable and $KU$ is well-defined. Since $-\Delta KU = U$ with $KU\in L^\infty(0,T;H^1(\Omega))$ and $K \partial_t U \in L^2(0,T;(H^1)'(\Omega))$, it follows from \eqref{PLe1} and Young's inequality that
\begin{align*}
\frac{1}{2} \frac{d}{dt} \|\nabla KU\|_2^2 & = \int_\Omega U K\partial_t U\ dx = \int_\Omega U \left( - u_1 e^{-v_1} + u_2 e^{-v_2} \right)\ dx\\
& = - \frac{1}{2} \int_\Omega U \left[ (u_1+u_2) \left( e^{-v_1} - e^{-v_2} \right) + \left( e^{-v_1} + e^{-v_2} \right) U \right]\ dx \\
& = - \frac{1}{2} \int_\Omega \left( e^{-v_1} + e^{-v_2} \right) U^2\ dx - \frac{1}{2} \int_\Omega (u_1+u_2) \frac{e^{-v_1} - e^{-v_2} }{\sqrt{e^{-v_1} + e^{-v_2}}} \sqrt{e^{-v_1} + e^{-v_2}} U\ dx \\
& \le \frac{1}{8} \int_\Omega (u_1+u_2)^2 \frac{\left| e^{-v_1} - e^{-v_2} \right|^2}{e^{-v_1} + e^{-v_2}}\ dx\ . 
\end{align*}
Since
\begin{equation*}
\frac{\left| e^{-X} - e^{-Y} \right|}{\sqrt{e^{-X} + e^{-Y}}} \le \frac{e^{-\min(X,Y)} |Y-X|}{\sqrt{e^{-X} + e^{-Y}}} \le |Y-X|\ , \qquad (X,Y)\in [0,\infty)^2\ ,
\end{equation*}
we further obtain
\begin{equation*}
\frac{d}{dt} \|\nabla KU\|_2^2 \le \frac{1}{4} \int_\Omega (u_1+u_2)^2 V^2\ dx\ ,
\end{equation*}
where $V := v_1-v_2$. Introducing $p:=2q/(q-2) \in (2,2^*]$ (notice that $p=2^*$ when $d\ge 3$), we infer from H\"older's inequality and Sobolev's embedding theorem that
\begin{equation}
\frac{d}{dt} \|\nabla KU\|_2^2 \le \frac{1}{4} \|u_1+u_2\|_{2p/(p-2)}^2 \|V\|_p^2 \le C_1 \|u_1+u_2\|_{q}^2 \|V\|_{H^1(\Omega)}^2\ . \label{PLu3}
\end{equation}

Also, from \eqref{PLe2} and Young's inequality,
\begin{align*}
\frac{1}{2} \frac{d}{dt} \|V\|_2^2 + \epsilon \|\nabla V\|_2^2 + \beta \|V\|_2^2 & = - \int_\Omega V \Delta KU\ dx = \int_\Omega \nabla V\cdot \nabla KU\ dx \\
& \le \frac{\epsilon}{2} \|\nabla V\|_2^2 + \frac{1}{2\epsilon} \|\nabla KU\|_2^2\ .
\end{align*}
Hence,
\begin{equation}
\frac{d}{dt} \|V\|_2^2 + \min\{\epsilon,2\beta\} \|V\|_{H^1(\Omega)}^2\le \frac{1}{\epsilon} \|\nabla KU\|_2^2\ . \label{PLu4}
\end{equation}

Now, let $T>0$ and set
\begin{equation}
M := \sup_{t\in [0,T]}\left\{ \|u_1(t)+u_2(t)\|_{q} \right\}\ , \label{PLu5}
\end{equation}
which is finite according to \eqref{PLu0}. Introducing $C_2 := \min\{\epsilon,2\beta\}/(C_1 M^2)$, we combine \eqref{PLu3}, \eqref{PLu4}, and \eqref{PLu5} to obtain
\begin{align*}
\frac{d}{dt} \left( C_2 \|\nabla KU\|_2^2 + \|V\|_2^2 \right) & \le C_1 C_2 M^2 \|V\|_{H^1(\Omega)}^2 - \min\{\epsilon,2\beta\} \|V\|_{H^1(\Omega)}^2 + \frac{1}{\epsilon} \|\nabla KU\|_2^2 \\
& = \frac{1}{\epsilon} \|\nabla KU\|_2^2 \le \frac{1}{\epsilon C_2} \left( C_2 \|\nabla KU\|_2^2 + \|V\|_2^2 \right)\ .
\end{align*}
After integration with respect to time, we end up with
\begin{equation*}
C_2 \|\nabla KU(t)\|_2^2 + \|V(t)\|_2^2 \le e^{t/(\epsilon C_2)} \left( C_2 \|\nabla KU(0)\|_2^2 + \|V(0)\|_2^2 \right)\ , \qquad t\in [0,T]\ .
\end{equation*}
Uniqueness readily follows from the previous inequality.
\end{proof}

\subsection{Regularity in spatial dimension one}

\begin{proof}[Proof of Proposition~\ref{reg_dim1}]
In spatial dimension one, the embedding $H^1(\Omega) \hookrightarrow C(\bar{\Omega})$ and Theorem~\ref{thm:globalexistence} implies that $v$ belongs to $L^\infty((0,T)\times \Omega)$ for all $T>0$, so that $1\ge e^{-v} \ge 1/C(T)>0$ for some $C(T)>1$. Then $u\in L^2((0,T)\times\Omega)$ by Theorem~\ref{thm:globalexistence} and, since
$$ \epsilon \Delta v = (\epsilon \Delta v - u + \beta v) + u - \beta v,$$
the $L^2$-regularity of $\epsilon \Delta v - u + \beta v$, $u$, and $v$ entail that $\Delta v\in L^2((0,T)\times\Omega)$. Hence, $v \in L^2(0,T;H^2(\Omega))$. 

Thanks to this additional regularity on $v$, for $u_0 \in L^r(\Omega)$, it is straightforward to derive the identity
\begin{align*}
\| u(t)\|_r^r & + r(r-1)\int_0^t \int_\Omega e^{-v} u^{r-2} |\partial_x u|^2 ~dx~ds \\
& =  \| u_0\|_r^r + r (r-1) \int_0^t \int_\Omega e^{-v} u^{r-1}\partial_x u\partial_x v ~dx~ds,
\end{align*}
from which we deduce with the help of Young's inequality, the nonnegativity of $v$, and the embedding $H^1(\Omega)\hookrightarrow C(\bar{\Omega})$ that
\begin{align*}  
\| u(t)\|_r^r &\leq  \| u_0\|_r^r +
\frac{r (r-1)}4 \int_0^t \int_\Omega  u^r  |\partial_x v|^2 ~dx~ds \\
& \leq \|u_0\|_r^r + \frac{r(r-1)}{4} \int_0^t \| \partial_x v\|_\infty^2 \| u\|_r^r ~ds \\
& \leq \| u_0\|_r^r + C r^2 \int_0^t \Vert v(s)\Vert_{H^2(\Omega)}^2 \| u\|_r^r ~ds .
\end{align*}
Since $s\mapsto \Vert v(s)\Vert_{H^2(\Omega)}^2 \in L^1(0,T)$, Gronwall's Lemma implies the boundedness of $\| u\|_r^r$ in $(0,T)$. Thus we have proved Proposition~\ref{reg_dim1}.
\end{proof}

\subsection{Regularity in spatial dimension two}
%%%%%%%%%%%%%%%%
%%%%%%%%%%%%%%%%

\newcounter{NumConst}

%%%%%%%%%%%%%%%%
%%%%%%%%%%%%%%%%

\refstepcounter{NumConst}\label{PLrcst1}

Let $T>0$. We recall that, from the entropy and duality estimates, there is $C_{\ref{PLrcst1}}(T)>0$ such that, for $t\in [0,T]$,
\begin{align}
\|\nabla Ku(t)\|_2^2 + \int_0^t \int_\Omega u^2 e^{-v}\ dxds & \le C_{\ref{PLrcst1}}(T)\ , \label{PLr1} \\
\int_\Omega u(t)\log u(t) - u(t) +1\ dx + \frac{\epsilon}{2} \|\nabla v(t)\|_2^2 + \frac{\beta}{2} \|v(t)\|_2^2 & \le C_{\ref{PLrcst1}}(T)\ , \label{PLr2} \\
\int_0^t \|\partial_t v\|_2^2\ ds + \int_0^t \int_\Omega u e^{-v} |\nabla(\ln{u}-v)|^2\ dxds & \le C_{\ref{PLrcst1}}(T)\ , \label{PLr3}
\end{align} 

In order to prove Proposition~\ref{P.PLr0} we need two lemmas providing intermediate estimates:
%%%%%%%%%%%%%%%%
\begin{lemma}\label{L.PLr1}
\refstepcounter{NumConst}\label{PLrcst2} There is $C_{\ref{PLrcst2}}(T)>0$ such that
\begin{equation*}
\|u(t)\|_2 + \|v(t)\|_\infty \le C_{\ref{PLrcst2}}(T)\ , \qquad t\in [0,T]\ .
\end{equation*}
\end{lemma}
%%%%%%%%%%%%%%%%

\begin{proof}
Set $w:= u e^{-v/2}$. It follows from \eqref{PLe1} that
\begin{align}
\frac{1}{2} \frac{d}{dt} \|u\|_2^2 & = \int_\Omega w e^{v/2} \Delta\left( w e^{-v/2} \right)\ dx = - \int_\Omega \nabla \left( w e^{v/2} \right) \cdot  \nabla\left( w e^{-v/2} \right)\ dx \nonumber \\
& = - \int_\Omega \left( \nabla w + \frac{w}{2} \nabla v \right) \cdot \left( \nabla w - \frac{w}{2} \nabla v \right)\ dx \nonumber \\
& = - \|\nabla w\|_2^2 + \frac{1}{4} \|w \nabla v\|_2^2\ . \label{PLr4}
\end{align}	
Next, by H\"older's inequality,
\begin{equation*}
\|w \nabla v\|_2^2 \le \|w\|_4^2 \|\nabla v\|_4^2\ . 
\end{equation*}
On the one hand, it follows from Gagliardo-Nirenberg's inequality that
\begin{equation*}
\|w\|_4 \le C \|w\|_{H^1(\Omega)}^{1/2} \|w\|_2^{1/2}\ .
\end{equation*}
On the other hand, we infer from \eqref{PLe1}, \eqref{PLr2}, and Gagliardo-Nirenberg's and Calderon-Zygmund's inequalities that
\begin{align*}
\|\nabla v\|_4 & \le C \left( \|D^2 v\|_2^{1/2} \|\nabla v\|_2^{1/2} + \|\nabla v\|_2 \right) \le C(T) \left( 1 +  \|\Delta v\|_2^{1/2} \right) \\
& \le C(T) \left( 1 + \|\partial_t v + \beta v - u\|_2 \right)^{1/2} \le C(T) \left( 1 + \|\partial_t v\|_2 + \|u\|_2 \right)^{1/2}\ .
\end{align*}
Combining the above estimates leads us to
\begin{equation*}
\|w \nabla v\|_2^2 \le C(T) \|w\|_{H^1(\Omega)} \|w\|_2 \left( 1 + \|\partial_t v\|_2 + \|u\|_2 \right).
\end{equation*}
Hence, after using Young's inequality and the property $w\le u$,
\begin{align}
\|w \nabla v\|_2^2 & \le 2 \|w\|_{H^1(\Omega)}^2 + C(T) \|w\|_2^2 \left( 1 + \|\partial_t v\|_2^2 + \|u\|_2^2 \right) \nonumber \\
& \le 2 \|\nabla w\|_2^2 + C(T) \|w\|_2^2 + C(T) \|\partial_t v\|_2^2 \|u\|_2^2 + C(T) \|w\|_2^2 \|u\|_2^2\ . \label{PLr5}
\end{align}
It follows from \eqref{PLr4} and \eqref{PLr5} that 
\begin{align*}
\frac{d}{dt} \|u\|_2^2 & \le - \|\nabla w\|_2^2 + C(T) \|w\|_2^2 + C(T) \left(  \|\partial_t v\|_2^2 + \|w\|_2^2 \right) \|u\|_2^2 \\
& \le C(T) \|w\|_2^2 + C(T) \left(  \|\partial_t v\|_2^2 + \|w\|_2^2 \right)\ \|u\|_2^2 .
\end{align*}
Gronwall's lemma then entails that
\begin{equation*}
\|u(t)\|_2^2 \le \left( \|u_0\|_2^2 + C(T) \int_0^t \|w(s)\|_2^2\ ds \right) \exp\left\{ C(T) \int_0^t \left( \|\partial_t v(s)\|_2^2 + \|w(s)\|_2^2\right)\ ds \right\}
\end{equation*}
for $t\in [0,T]$ and we deduce from \eqref{PLr1} and \eqref{PLr3} that
\begin{equation}
\|u(t)\|_2 \le C(T)\ , \qquad t\in [0,T]\ . \label{PLr6}
\end{equation}
We next exploit the regularizing properties of the semigroup in $L^2(\Omega)$ associated with the linear operator $-\epsilon \Delta + \beta I$ supplemented with homogeneous Neumann boundary conditions to derive from \eqref{PLe2} that, for $t\in [0,T]$,
\begin{equation*}
\|v(t)\|_\infty \le C \|v_0\|_\infty + C \int_0^t (t-s)^{-1/2} e^{-\beta (t-s)} \|u(s)\|_2\ ds \le C \|v_0\|_\infty + C \sup_{s\in [0,T]}\{\|u(s)\|_2\}\ .
\end{equation*}
Combining the above inequality and \eqref{PLr6} completes the proof.
\end{proof}

%%%%%%%%%%%%%%%%
\begin{lemma}\label{L.PLr2}
\refstepcounter{NumConst}\label{PLrcst3} Let $p\in (2,\infty)$. There is $C_{\ref{PLrcst3}}(T,p)>0$ such that
\begin{equation*}
\|u(t)\|_p \le C_{\ref{PLrcst3}}(T,p)\ , \qquad t\in [0,T]\ .
\end{equation*}
\end{lemma}
%%%%%%%%%%%%%%%%

\begin{proof}
\refstepcounter{NumConst}\label{PLrcst4} Introducing $C_{\ref{PLrcst4}}(T) := e^{-C_{\ref{PLrcst2}}(T)}>0$, we infer from Lemma~\ref{L.PLr1} that
\begin{equation}
e^{-v(t,x)}\ge C_4(T)\ , \qquad (t,x)\in [0,T]\times \Omega\ . \label{PLr7}
\end{equation}

Next, let $p\in (2,\infty)$. By \eqref{PLe1}, \eqref{PLe2}, and \eqref{PLr7}, 
\begin{align*}
\frac{d}{dt} \|u\|_p^p & = - p (p-1) \int_\Omega u^{p-2} \nabla u \cdot \nabla\left( u e^{-v} \right)\ dx \\
& = - \frac{4(p-1)}{p} \int_\Omega e^{-v} |\nabla u^{p/2}|^2\ dx + (p-1) \int_\Omega e^{-v} \nabla u^p \cdot \nabla v\ dx \\
& \le - 2 C_{\ref{PLrcst4}}(T) \left\| \nabla u^{p/2} \right\|_2^2 - (p-1) \int_\Omega e^{-v} u^p \Delta v\ dx + (p-1) \int_\Omega e^{-v} u^p |\nabla v|^2\ dx \\
& \le - 2 C_{\ref{PLrcst4}}(T) \left\| \nabla u^{p/2} \right\|_2^2 + (p-1) \int_\Omega u^p \left( |\Delta v| + |\nabla v|^2 \right)\ dx\ .
\end{align*}
A further use of H\"older's inequality gives
\begin{equation}
\frac{d}{dt} \|u\|_p^p \le - 2 C_{\ref{PLrcst4}}(T) \left\| \nabla u^{p/2} \right\|_2^2 + (p-1) \left( \|\Delta v\|_2 + \|\nabla v\|_4^2 \right) \|u\|_{2p}^p\ . \label{PLr8}
\end{equation}
To estimate the last term on the right hand side of \eqref{PLr8}, we observe that \eqref{PLe2}, \eqref{PLr2}, \eqref{PLr3}, and Lemma~\ref{L.PLr1} imply that
\begin{equation}
\int_0^T \|\Delta v(s)\|_2^2\ ds \le C \int_0^T \left( \|\partial_t v(s)\|_2^2 + \|v(s)\|_2^2 + \|u(s)\|_2^2 \right)\ ds \le C(T)\ , \label{PLr9}
\end{equation}
while \eqref{PLr2} and Gagliardo-Nirenberg's and Calderon-Zygmund's inequalities guarantee that
\begin{align*}
\int_0^T \|\nabla v(s)\|_4^4\ ds & \le C \int_0^T \left( \|D^2 v(s)\|_2^2 \|\nabla v(s)\|_2^2 + \|\nabla v(s)\|_2^2 \right)\ ds \\
& \le C(T) \int_0^T \left( 1 + \|\Delta v(s)\|_2^2 \right)\ ds\ . 
\end{align*}
Combining \eqref{PLr9} and the previous inequality gives
\begin{equation}
\int_0^T \|\nabla v(s)\|_4^4\ ds \le C(T)\ . \label{PLr10}
\end{equation}
We next infer from Gagliardo-Nirenberg's inequality that
\begin{align*}
\|u\|_{2p}^p & = \left\| u^{p/2} \right\|_4^2 \le C \left( \left\| \nabla u^{p/2} \right\|_2 \left\| u^{p/2} \right\|_2 + \left\| u^{p/2} \right\|_2^2 \right) \\
& = C \left( \left\| \nabla u^{p/2} \right\|_2 \|u\|_p^{p/2} + \|u\|_p^p \right)\ ,
\end{align*}
which, together with \eqref{PLr8} and Young's inequality, leads us to 
\begin{align*}
\frac{d}{dt} \|u\|_p^p & \le - 2 C_{\ref{PLrcst4}}(T) \left\| \nabla u^{p/2} \right\|_2^2 + (p-1) C \left( \|\Delta v\|_2 + \|\nabla v\|_4^2 \right) \left( \left\| \nabla u^{p/2} \right\|_2 \|u\|_p^{p/2} + \|u\|_p^p \right) \\
& \le - C_{\ref{PLrcst4}}(T) \left\| \nabla u^{p/2} \right\|_2^2 + (p-1)^2 C \left( \|\Delta v\|_2 + \|\nabla v\|_4^2 \right)^2 \|u\|_p^p \\
& \qquad + (p-1) C \left( \|\Delta v\|_2 + \|\nabla v\|_4^2 \right) \|u\|_p^p \\
& \le (p-1)^2 C \left( 1 + \|\Delta v\|_2^2 + \|\nabla v\|_4^4 \right) \|u\|_p^p \ .
\end{align*}
By Gronwall's lemma,
\begin{equation*}
\|u(t)\|_p^p \le \|u_0\|_p^p \exp\left\{ (p-1)^2 C(T) \int_0^t \left( 1 + \|\Delta v(s)\|_2^2 + \|\nabla v(s)\|_4^4 \right)\ ds \right\}\ , \qquad t\in [0,T]\ ,
\end{equation*}
and we complete the proof with the help of \eqref{PLr9} and \eqref{PLr10}.
\end{proof}

\begin{proof}[Proof of Proposition~\ref{P.PLr0}]
Let $p>2$. We first infer from \eqref{PLe2}, Lemma~\ref{L.PLr2}, and the regularizing properties of the semigroup in $L^p(\Omega)$ associated with the linear operator $-\epsilon \Delta + \beta I$ supplemented with homogeneous Neumann boundary conditions that, for $t\in [0,T]$, 
\refstepcounter{NumConst}\label{PLrcst5}
\begin{align} \nonumber
\|\nabla v(t) \|_p
&\le C \|\nabla v_0\|_p + C \int_0^t (t-s)^{-1/2} e^{-\beta (t-s)}\|u(s)\|_p\ ds \\
&\le C (1+ \sup_{s\in(0,t)}\|u(s)\|_p) \le C_{\ref{PLrcst5}}(T,p)\ . \label{PLr11}
\end{align}
Next an alternative formulation of \eqref{PLe1} reads
\begin{equation*}
\partial_t u = \mathrm{div} (D \nabla u + \mathbf{F}) \;\text{ in }\; (0,T)\times \Omega\ , 
\end{equation*}
with
\begin{equation*}
D := e^{-v} \;\text{ and }\; \mathbf{F} := - u e^{-v} \nabla v\ .
\end{equation*}
Now, $u\in L^\infty(0,T;L^2(\Omega))$ and $D\ge e^{-C_{\ref{PLrcst2}}(T)}> 0$ in $(0,T)\times \Omega$  by Lemma~\ref{L.PLr1}, while $\mathbf{F}\in L^\infty(0,T;L^4(\Omega;\mathbb{R}^2))$ by \eqref{PLr11} and Lemma~\ref{L.PLr2}. We are then in a position to apply \cite[Lemma~A.1]{TaoWinklerBound} to conclude that $u\in L^\infty((0,T)\times\Omega)$. Since the boundedness of $v$ follows from Lemma~\ref{L.PLr1}, the proof is complete.
\end{proof}

\section{Conclusion and outlook}\label{sec:conclusion}

We have analyzed a model for chemotactic pattern formation with a local sensing mechanism instead of the commonly used gradient sensing. We have shown that even for arbitrary large mass there exists a weak-strong solution for arbitrary times, i.e., the blow-up is delayed to infinity compared to the well-studied case of gradient sensing. Moreover, under slightly higher integrability of the initial value we obtain uniqueness of the solution and some dimension-dependent regularity results.

The results of this paper indicate the interest of a more detailed investigation of processes with local sensing, respectively the corresponding gradient structure of the form
\begin{equation}
\partial_t  u = \Delta e^{\partial_u{\cal E}(u)} ,
\end{equation}
which also appears for different energies, see, e.g., \cite{gao2019gradient} for the Dirichlet energy. Note that such equations possess a formal entropy dissipation of the form
$$ \frac{d}{dt} {\cal E}(u) =  - \int_\Omega e^{\partial_u{\cal E}(u)} |\nabla \partial_u{\cal E}(u)|^2~dx, $$
as well as duality estimates for the $\Hdual(\Omega)$-norm, which may be crucial for their analysis.  

Concerning the link to the models with gradient sensing it might be relevant to study some intermediate cases. A natural interpolation is obtained when one starts from a microscopic model based on a jump process with rate combining local and gradient sensing. To illustrate this, let $V=V(x)$ be a given potential (or in this case a chemotactic signal) and consider for $\theta \in [0,1]$ the jump process on a one-dimensional grid ($h$ being the size of the grid) with the probability of jump from node $i$ to node $j$ given by
$$ k_{i,j} = \exp\left[ -\theta V(x_i) - (1-\theta) h (V(x_i)-V(x_j))\right], $$
if $|i-j|=1$, while $k_{i,j} = 0$ for $|i-j|>1$. Using a standard scaling of $k_{i,j}$ as the grid size tends to zero, we obtain the equation
$$ \partial_t u = \nabla\cdot \lk e^{-\theta V} (\nabla u - u\nabla V)\rk = \nabla\cdot \lk e^{-\theta V}u\nabla \log (u e^{-V})\rk = \nabla\cdot \lk e^{(1-\theta) V} \nabla (u e^{-V})\rk ,$$
which is the standard Fokker-Planck equation for $\theta =0$ and the equation for local sensing for $\theta =1$. For any $\theta \in [0,1]$ the system has the same entropy, but duality estimates only hold for $\theta =1$.

\section*{Acknowledgements}

This work has been supported by the DAAD-PROCOPE project {\em Nonlinear Cross-Diffusion Systems}. Part of this work was done while PhL and AT enjoyed the hospitality of the Department Mathematik, FAU Erlangen-N\"urnberg. The authors thank Antonio Esposito (FAU Erlangen-N\"urnberg) for useful suggestions on the manuscript.

\bibliographystyle{abbrv}
\bibliography{ChemotaxisBIB}

%%%%%%%%%%%%%%%%
%%%%%%%%%%%%%%%%

\end{document}